\documentclass[onefignum,onetabnum]{siamart190516}

\usepackage{amsfonts}
\usepackage{graphicx}
\usepackage{epstopdf}
\usepackage{algorithmic}
\ifpdf
  \DeclareGraphicsExtensions{.eps,.pdf,.png,.jpg}
\else
  \DeclareGraphicsExtensions{.eps}
\fi

\newsiamremark{remark}{Remark}
\newsiamremark{hypothesis}{Hypothesis}
\crefname{hypothesis}{Hypothesis}{Hypotheses}
\newsiamthm{claim}{Claim}

\headers{Conservation laws with discontinuous flux}{J. Badwaik and A. M. Ruf}

\title{Convergence rates of monotone schemes for conservation laws with discontinuous flux\thanks{Submitted August 23, 2019.
\funding{Both authors have received funding from the European Union’s Framework Programme for Research and Innovation Horizon 2020 (2014-2020) under the Marie Sk{\l}odowska-Curie Grant Agreement No. 642768. In addition, J. Badwaik was supported by the Priority Programme~1648 of the German Science Foundation and A. M. Ruf has received funding from NFR-DAAD grant No.~281268.}}}

\author{Jayesh Badwaik\thanks{Department of Mathematics, University of W\"urzburg, Germany (\email{badwaik.jayesh@gmail.com}).}
\and Adrian M. Ruf\thanks{Seminar for Applied Mathematics, ETH Z\"urich, Switzerland
    (\email{adrian.ruf@sam.math.ethz.ch}).}
}

\usepackage{amsopn}

\ifpdf
\hypersetup{
  pdftitle={Convergence rates for conservation laws with discontinuous flux},
  pdfauthor={J. Badwaik and A. M. Ruf}
}
\fi

\usepackage{units}
\newcommand{\R}{\mathbb{R}}

\newcommand{\Z}{\mathbb{Z}}

\newcommand{\cell}{\mathcal{C}}
\newcommand{\sign}{\operatorname{sign}}

\newcommand{\Lone}{\rm{L}^1}

\newcommand{\TV}{\operatorname{TV}}
\newcommand{\BV}{\mathrm{BV}}

\newcommand{\Dx}{{\Delta x}}
\newcommand{\Dt}{{\Delta t}}
\newcommand{\hf}{{\unitfrac{1}{2}}}

\newcommand{\jphf}{{j+\hf}}
\newcommand{\jmhf}{{j-\hf}}

\newcommand{\Dplust}{D_{+}^t}
\newcommand{\Dminus}{D_{-}}

\newcommand*\diff{\mathop{}\!\mathrm{d}}

\usepackage{tikz}
\usepackage{pgfplots}
\pgfplotsset{width=.38\textwidth,compat=1.12}

\usepackage{subfig}

\definecolor{skyblue1}{rgb}{0.447,0.624,0.812}
\definecolor{plum1}{rgb}{0.678,0.498,0.659}
\definecolor{scarletred1}{rgb}{0.937,0.161,0.161}

\newtheorem*{maintheorem}{Main Theorem}

\usepackage{siunitx}
\usepackage{booktabs}

\usepackage{bbm}

\begin{document}

\maketitle

\begin{abstract}
  We prove that a class of monotone finite volume schemes for scalar conservation laws with discontinuous flux converge at a rate of $\sqrt{\Delta x}$ in $\mathrm{L}^1$, whenever the flux is strictly monotone in $u$ and the spatial dependency of the flux is piecewise constant with finitely many discontinuities.
  We also present numerical experiments to illustrate the main result.
  To the best of our knowledge, this is the first proof of any type of convergence rate for numerical methods for conservation laws with discontinuous, nonlinear flux.

  Our proof relies on convergence rates for conservation laws with initial and boundary value data. Since those are not readily available in the literature we establish convergence rates in that case \emph{en passant} in the Appendix.
\end{abstract}



\section{Introduction}

We prove a convergence rate for a class of monotone, upwind-type finite volume schemes for scalar conservation laws with \emph{discontinuous} flux of the form
\begin{gather}
\begin{aligned}
  u_t + f(k(x),u)_x = 0,& &&(x,t)\in \R\times(0,T),\\
  u(x,0) = u_0(x),& &&x\in \R.
\end{aligned}
\label{conservation law}
\end{gather}
Here, we assume that the flux $f$ is \emph{strictly monotone} in $u$ and has a \emph{discontinuous} spatial dependency through the coefficient $k$ which is \emph{piecewise constant} with finitely many discontinuities.


\begin{maintheorem}
Let $f$ be strictly monotone in $u$ in the sense that $f_u >0$, $k$ piecewise constant with finitely many discontinuities, and $u_0\in(\mathrm{L}^1\cap\BV)(\R)$. Then all monotone finite volume methods with the upwind property which obey the discrete Rankine--Hugoniot condition across the discontinuities of $k$ converge at a rate of $\sqrt{\Dx}$ to the unique entropy solution of the conservation law~\eqref{conservation law}.
\end{maintheorem}
The full theorem is stated in Section~\ref{sec: statement and proof of the main theorem}. Our proof uses the Rankine--Hugoniot condition at the discontinuities of $k$ to break down the problem into finitely many initial-boundary value problems for each of which we will prove a convergence rate using the classical `doubling of variables' technique.


\subsection{Background on conservation laws with discontinuous fluxes}
Problem~\eqref{conservation law} is of great practical interest in several areas of physics and engineering. In particular, it arises in modeling traffic flow on highways with changing road conditions (see \cite{lighthill1955kinematic}), in the modeling of two-phase flow in a porous medium (see \cite{gimse1992solution,risebro1991front}), and in modeling sedimentation processes (see \cite{diehl1996conservation,burger2003front}).

The flux in~\eqref{conservation law} depends on the space variable through a coefficient $k$ which may be discontinuous. The dependence can for example be of the additive type, i.e., $f(k(x),u)=f(u)-k(x)$ (see \cite{greenberg1997analysis}), or of the multiplicative type, i.e., $f(k(x),u)=k(x)f(u)$ which is more common (see e.g. \cite{Towers1}). However, for the sake of generality we do not assume any particular algebraic structure of the flux $f(k(x),u)$ here.
The case we consider in this paper where $k$ is piecewise constant with finitely many discontinuities corresponds to switching from one $u$-dependent flux function, $f^{(i-1)}$, to another, $f^{(i)}$, across a discontinuity $\xi_i$ of $k$.
When $k$ has just one discontinuity -- the so-called `two flux' case -- given by
\begin{equation}
  u_t + (H(x)f(u) + (1-H(x))g(u))_x = 0
  \label{two flux case}
\end{equation}
where $H$ is the Heaviside function was studied in a series of papers by Adimurthi, Mishra, and Gowda (see \cite{Sid2005,adimurthi2005optimal,ADIMURTHI2007310} and references therein). Most notably, in \cite{adimurthi2005optimal}, the authors showed existence of infinitely many $\mathrm{L}^1$-stable semi-groups of solutions to~\eqref{two flux case}.
We remark that, because of the assumption that $k$ is piecewise constant, the convergence rate of monotone schemes for~\eqref{two flux case} will be the building block for the general case of~\eqref{conservation law}.

Equations of type~\eqref{conservation law} have been dealt with extensively in the literature from a purely academic point of view as well as with a specific application in mind. In \cite{gimse1991riemann,gimse1992solution}, Gimse and Risebro calculated solutions for the Riemann problem assuming convexity of the flux in $u$ and used the solutions to show existence of a weak solution for the general Cauchy problem with a front tracking algorithm. Other results based on the front tracking algorithm were obtained in \cite{klingenberg1995convex}, \cite{klausen1999stability}, \cite{BAITI1997161}, \cite{klingenberg2001stability}, \cite{burger2003front},  and in \cite{coclite2005conservation} with a time-dependent discontinuous coefficient.
Out of the aforementioned results, we want to highlight \cite{BAITI1997161} from Baiti and Jensen who proved existence and uniqueness of entropy solutions in the case that the flux is strictly monotone in $u$ which is the case we consider in this paper as well.

The first results for finite volume schemes for~\eqref{conservation law} (assuming a multiplicative spatial dependency) were obtained by Towers in \cite{Towers1,Towers2}. Specifically, in \cite{Towers1}, the author developed staggered versions of the Godunov and Engquist--Osher schemes for the case where $f$ is convex in $u$ and $k$ is strictly positive. In \cite{Towers2} similar results were proved for the case of non-convex fluxes.
In \cite{karlsen2002upwind}, Karlsen, Risebro, and Towers studied~\eqref{conservation law} with an added degenerate parabolic term using an Engquist--Osher-type scheme and in \cite{karlsen2002nonlinear} the authors proved existence of the vanishing viscosity limit using compensated compactness.
In \cite{karlsen2004convergence}, Karlsen and Towers showed convergence of the Lax--Friedrichs scheme for~\eqref{conservation law} (with a time-dependent discontinuous coefficient). They were able to handle very general fluxes and sign-changing coefficients by using compensated compactness.

A general framework for well-posedness of~\eqref{two flux case} was proposed by Andreianov, Karlsen, and Risebro in \cite{andreianov2011theory}.

Lastly, we want to point out that the monotonicity assumption, $ f_u >0$ we use in this paper implies that the equivalent system
\begin{gather*}
  \begin{aligned}
    u_t + f(k,u)_x =0,&\\
    k_t=0&
  \end{aligned}
\end{gather*}
is hyperbolic and not resonant, see \cite{Towers1,Towers2,klausen1999stability,klingenberg1995convex,klingenberg2001stability}.

\subsection{Background on convergence rates}
When dealing with numerical methods for~\eqref{conservation law}, where an approximate solution $u_\Dt$ depends on a grid discretization parameter $\Dx$, having a provable bound of the type
\begin{equation*}
	\|u(T)-u_\Dt(T)\|_{\mathrm{L}^1(\R)} \leq C \Dx^r,
\end{equation*}
-- specifying how fast the numerical scheme converges -- is highly desirable. Specifically, convergence rates can be used for a posteriori error based mesh adaptation \cite{venditti2000adjoint} and optimal design of multilevel Monte Carlo methods \cite{badwaik2019multilevel}.
To this date the only result concerning convergence rates of finite volume schemes for conservation laws with discontinuous flux is due to Wen and Jin and pertains the most basic case of the linear advection equation with piecewise constant wave speed that changes across a single discontinuity, \cite{wen2008convergence}.
So far, in the nonlinear case, convergence rates of finite volume schemes for~\eqref{conservation law} are only available in the absence of a spatial dependency, i.e., $k$ being constant.
The main difficulty in obtaining convergence rates when the flux has a discontinuous spatial dependency is that in this case the classical `doubling of variables' technique (see \cite{kruvzkov1970first}) involves both, terms with $k(x)$ and terms with $k(y)$.

In the case of a spatially independent flux the seminal paper by Kuznetsov \cite{kuznetsov76} shows that monotone schemes converge towards the entropy solution of~\eqref{conservation law} without spatial dependency at a rate of $\mathcal{O}(\sqrt{\Dx})$ in $\mathrm{L}^1$. This rate was proved for initial data in $(\mathrm{L}^1\cup\mathrm{BV})(\R)$, and in this generality the rate $\mathcal{O}(\sqrt{\Dx})$ is in fact optimal, as was shown by \c{S}abac in \cite{sabac}. There are certain classes of initial data for which higher orders of convergence for monotone schemes have been shown, e.g. Teng and Zhang \cite{zhang_teng} showed a convergence rate of $\mathcal{O}(\Dx)$ for the case of piecewise constant initial data. See also \cite{Ruf2019} for a more comprehensive overview of convergence rate results for~\eqref{conservation law} without spatial dependency.

An alternative approach to convergence rates in the case where the flux only depends on $u$ was initiated by Nessyahu, Tadmor, and Tassa \cite{nessyconv92,nessyconv}. The authors used the Wasserstein distance instead of the $\mathrm{L}^1$ norm and were able to show that a large class of monotone schemes converge at a rate of $\mathcal{O}(\Dx)$ in the Wasserstein distance for $\mathrm{Lip}^+$-bounded, compactly supported initial data. This rate was recently proved to be optimal by Ruf, Sande, and Solem \cite{Ruf2019}.

Since the proof of our main theorem makes use of convergence rates for conservation laws on bounded domains, it is worth mentioning that Ohlberger and Vovelle claimed a convergence rate of $\mathcal{O}(\Dx^{\unitfrac{1}{3}})$ for conservation laws with initial and boundary data in one dimension \cite[p. 135]{ohlberger2006error}. In our specific case of a strictly monotone flux however, we are able to prove a better rate of $\mathcal{O}(\sqrt{\Dx})$.

\subsection{Outline of the paper}

We have organized the paper in the following way. In Section~\ref{sec: Preliminaries}, we will define entropy solutions of~\eqref{conservation law} and show that
-- when restricted to a subdomain between two neighboring discontinuities of $k$ -- they are entropy solutions of a certain initial boundary value problem with spatially independent flux.
Here the respective boundary datum is given through the Rankine--Hugoniot condition across a discontinuity of $k$. In Section~\ref{sec: The numerical scheme}, we describe our finite volume scheme and show that we can establish a convergence rate of our numerical method for~\eqref{conservation law} by proving a convergence rate for each of those initial-boundary value problems. In Section~\ref{sec: Convergence rates for fluxes with one discontinuity}, we start by considering just one discontinuity of $k$, i.e., Equation~\eqref{two flux case}, and proving a convergence rate on $\R^-$ and $\R^+$ separately. Section~\ref{sec: statement and proof of the main theorem} contains the statement and proof of the main result where we use the translation invariance of conservation laws and the results of the previous section in our main proof. Section~\ref{sec: Numerical experiments} describes numerical experiments that illustrate our convergence rate result as well as the class of fluxes that is covered by our theory. In Section~\ref{sec: Conclusion}, we summarize the findings of this paper and provide an outlook.
Lastly, in Appendix~\ref{app: Convergence rates for general IBVPs}, we show that -- with minimal changes -- our results can be applied to general initial-boundary value problems where the prescribed boundary datum is arbitrary.

\section{Preliminaries}\label{sec: Preliminaries}

Throughout this paper, we will assume that the initial datum $u_0$ is integrable, bounded, and of finite total variation, i.e., $u_0 \in(\mathrm{L}^1\cap\BV)(\R)$, and that $f$ is strictly monotone in $u$, i.e., $f_u\geq \alpha>0$. Further, we will denote the discontinuities of $k$ as $\xi_1,\ldots,\xi_N$ and the interval between two adjacent discontinuities as $D_i = (\xi_i,\xi_{i+1})$, $i=0,\ldots, N$. Here, we used the notation $\xi_0=-\infty$ and $\xi_{N+1}=\infty$. Then we can write
\begin{equation*}
  f(k(x),\cdot) =: f^{(i)}(\cdot)\qquad \text{for }x\in D_i.
\end{equation*}
We will consider entropy solutions of Equation \eqref{conservation law} in the following sense.
\begin{definition}[Entropy solution]\label{def: entropy solution}
  We say $u\in\mathcal{C}([0,T];\mathrm{L}^1(\R))\cap\mathrm{L}^\infty((0,T)\times\R)$ is an entropy solution of Equation~\eqref{conservation law} if for all $c\in\R$
  \begin{multline*}
    \sum_{i=0}^N \bigg( \int_0^T \int_{D_i} (|u-c_i|\varphi_t + \sign(u-c_i)(f^{(i)}(u)-f^{(i)}(c_i))\varphi_x) \diff x\diff t \\
    -\int_{D_i}|u(x,T)-c_i|\varphi(x,T)\diff x +\int_{D_i}|u_0(x)-c_i|\varphi(x,0))\diff x\\
    - \int_0^T \sign(u(\xi_{i+1}-,t)-c_i)(f^{(i)}(u(\xi_{i+1}-,t))-f^{(i)}(c_i))\varphi(\xi_{i+1},t)\diff t \\
    + \int_0^T \sign(u(\xi_i +,t)-c_i)(f^{(i)}(u(\xi_i +,t))-f^{(i)}(c_i))\varphi(\xi_i,t)\diff t \bigg) \geq 0.
  \end{multline*}
  for all nonnegative $\varphi\in\mathcal{C}^\infty(\R\times[0,T])$. Here, the $c_i$ are given by $c_0 :=c$ and
  \begin{equation}
    c_{i+1}=(f^{(i+1)})^{-1}(f^{(i)}(c_i)) \qquad \text{for } i=1,\ldots,N.
    \label{definition c}
  \end{equation}
\end{definition}
\begin{remark}
  Note that, due to the monotonicity of the fluxes $f^{(i)}$, the inverse of $f^{(i)}$ used in \eqref{definition c} and throughout this paper exists.
\end{remark}
\begin{remark}
  Note that existence and uniqueness of entropy solutions of Equation~\eqref{conservation law} are guaranteed by the theory developed by Baiti and Jensen in \cite{BAITI1997161} using adapted entropies in the sense above (cf. also \cite{audusse2005uniqueness} where adapted entropies are used as well). In particular, the traces in Defintion~\ref{def: entropy solution} are well defined (cf. \cite[Remark 2.3]{andreianov2011theory})
\end{remark}

\begin{remark}
  Like for conservation laws without (discontinuous) spatial dependency of the flux, a Rankine--Hugoniot-type argument shows that weak solutions of~\eqref{conservation law} necessarily satisfy the Rankine--Hugoniot condition across all discontinuities~$\xi_i$, i.e.,
  \begin{equation}
    f^{(i-1)}(u(\xi_i-,t)) = f^{(i)}(u(\xi_i+,t)).
    \label{continuous Rankine--Hugoniot condition}
  \end{equation}
\end{remark}

The following observation is at the heart of the proof of the main result. The entropy solution $u$ of~\eqref{conservation law} can be decomposed as $u=\sum_{i=0}^N u^{(i)}$ where $u^{(i)}:=u \mathbbm{1}_{D_i\times [0,T]}$ such that $u^{(0)}$ solves
\begin{gather}
\begin{aligned}
  u^{(0)}_t + f^{(0)}(u^{(0)})_x = 0,& &&(x,t)\in D_0\times(0,T),\\
  u^{(0)}(x,0) = u_0(x),& &&x\in D_0
\end{aligned}
\label{conservation law on D0}
\end{gather}
and $u^{(i)}$ solves
\begin{gather}
\begin{aligned}
  u^{(i)}_t + f^{(i)}(u^{(i)})_x = 0,& &&(x,t)\in D_i\times(0,T),\\
  u^{(i)}(x,0) = u_0(x),& &&x\in D_i,\\
  u^{(i)}(\xi_i +,t) = (f^{(i)})^{-1}\left(  f^{(i-1)}(u^{(i-1)}(\xi_i -,t)) \right),& &&t\in(0,T)
\end{aligned}
\label{conservation law on Di}
\end{gather}
for $i=1,\ldots,N$ (cf. Definitions~\ref{def: entropy solution on R-} and~\ref{def: entropy solution on R+} below).
Note that the boundary condition on the domain $D_i$, $i=1,\ldots,N$, given by the last line of \eqref{conservation law on Di} reflects the Rankine--Hugoniot condition \eqref{continuous Rankine--Hugoniot condition}.

Conversely, if $u^{(0)}$ is the entropy solution of \eqref{conservation law on D0} on $D_0$ and $u^{(i)}$ is the entropy solution of \eqref{conservation law on Di} on $D_i$ for $i=1,\ldots,N$, then the composite function $u:=\sum_{i=0}^N u^{(i)}$ is the entropy solution of \eqref{conservation law} in the sense of Definition~\ref{def: entropy solution}. This can be seen by adding the entropy inequalities of $u^{(i)}$ and choosing the respective constant in each entropy inequality in accordance with \eqref{definition c}.

In the remainder of the paper, we will construct a numerical scheme which satisfies the Rankine--Hugoniot condition \eqref{continuous Rankine--Hugoniot condition} (or, equivalently, the last line of \eqref{conservation law on Di}) across the discontinuities of $k$ on the discrete level. This scheme, when restricted to the subdomain $D_i$, will converge towards the entropy solution on $D_i$. The discrete Rankine--Hugoniot condition will then allow us to break down the problem of finding a convergence rate on the whole real line to finding convergence rates on each of the subdomains $D_i$.

\section{The numerical scheme}\label{sec: The numerical scheme}
We discretize the domain $\R\times[0,T]$ using the spatial and temporal grid discretization parameters $\Dx$ and $\Dt$. The resulting grid cells then are $\cell_j = (x_\jmhf,x_\jphf)$ and $\cell^n=(t^n,t^{n+1})$ for points $x_\jphf$, such that $x_\jphf-x_\jmhf=\Dx$, $j\in\Z$,  and $t^n=n\Dt$ for $n=0,\ldots,M+1$. Note that $T=(M+1)\Dt$. Further we write $\cell_j^n$ to denote the rectangle $\cell_j\times\cell^n$.

In the following we will assume that the grid is aligned in such a way that all discontinuities of $k$ lie on cell interfaces, i.e., $\xi_i = x_{P_i-\hf}$ for some integers $P_i$, $i=1,\ldots, N$. In general this can easily be achieved by considering a globally non-uniform grid that is uniform on each $D_i$ and then taking $\Dx = \max_{i=0,\ldots, N} \Dx_i$ where $\Dx_i$ is the grid discretization parameter in $D_i$. For simplicity however, here we will assume that the grid is uniform on the whole real line.


Further, we will consider two-point numerical fluxes $F(u,v)$ that have the upwind property such that if $f'\geq 0$ then $F(u,v)=f(v)$. Such fluxes include the upwind flux, the Godunov flux, and the Engquist--Osher flux.
Thus, the numerical scheme we will analyze is the following:
\begin{gather}
\begin{aligned}
  u_j^{n+1} = u_j^n - \lambda \left( f^{(i)}(u_j^n) - f^{(i)}(u_{j-1}^n) \right),& &&n\geq 0,~P_i<j<P_{i+1},~0\leq i\leq N\\
  u_j^0 = \frac{1}{\Dx} \int_{\cell_j} u_0(x)\diff x,& &&j\in\Z,\\
  u_{P_i}^{n+1} = (f^{(i)})^{-1}\left(f^{(i-1)}(u_{P_i -1}^{n+1})\right),& &&n\geq 0,~0<i\leq N
\end{aligned}
\label{numerical scheme on R}
\end{gather}
where $\lambda=\Dt/\Dx$.
We assume that the grid discretization parameters satisfy the CFL condition
\begin{equation}
  \max_i \max_u (f^{(i)})'(u) \lambda  \leq 1.
  \label{CFL condition}
\end{equation}

Note that the last line of \eqref{numerical scheme on R} represents a discrete version of the Rankine--Hugoniot condition \eqref{continuous Rankine--Hugoniot condition}. Here, we use the ghost cells $\cell_{P_i}$, $i=1,\ldots,N$ to explicitely enforce the Rankine--Hugoniot condition on the discrete level. While this makes the numerical scheme \eqref{numerical scheme on R} non-conservative, the convergence result in this paper, coupled with the fact that the limit is conservative, shows that the contribution of the non-conservative part of the scheme vanishes in the limit.

To get a convergence rate of the numerical scheme~\eqref{numerical scheme on R} we decompose the entropy solution $u$ as $u=\sum_{i=0}^N u^{(i)}$ where $u^{(i)}$, $i=0,\ldots,N$, are the respective entropy solutions on $D_i$ and the numerical solution $u_\Dt$ as $\sum_{i=0}^N u^{(i)}_\Dt$ where
\begin{equation*}
  u^{(i)}_\Dt(x,t)=\begin{cases}
    u_j^n &\text{if } (x,t)\in\cell_j^n\subset D_i\times\cell^n,\\
    0 &\text{otherwise}
  \end{cases}
\end{equation*}
Then we have
\begin{equation*}
 \|u(T) - u_{\Dt}(T)\|_{\mathrm{L}^1(\R)} = \sum_{i=0}^N \|u^{(i)}(T) - u^{(i)}_\Dt(T)\|_{\mathrm{L}^1(D_i)}
\end{equation*}
and the problem of finding a convergence rate for $u_\Dt$ can be broken down to finding convergence rates for $u^{(i)}_\Dt$ on each of the subdomains $D_i$.
In the following sections, we will show that
\begin{equation*}
  \|u^{(i)}(T) - u^{(i)}_\Dt(T)\|_{\mathrm{L}^1(D_i)} \leq C \sqrt{\Dx}.
\end{equation*}
Note that convergence of the numerical scheme~\eqref{numerical scheme on R} towards the entropy solution of \eqref{conservation law} follows from our convergence rate estimate. At this point we want to point out that instead of assuming $f_u>0$ our proof can readily be adapted for the case $f_u<0$.

\section{Convergence rates for fluxes with one discontinuity}\label{sec: Convergence rates for fluxes with one discontinuity}
We will first consider the case where $k$ has just two constant values separated by a discontinuity $\xi_1$ and for ease of notation we will assume that $\xi_1=0$. Further, we will denote the flux left of $\xi_1$ as $g$ and right of $\xi_1$ as $f$.
In order to get a convergence rate for problem~\eqref{conservation law} we will derive convergence rates on $D_0 = \R^-$, on $D_1=\R^+$, and on $(0,L)$ for $L>0$.

\subsection{Convergence rate estimates on $\R^-$}\label{subsec: Convergence rates on R-}
As a first step we consider the initial value problem
\begin{gather}
\begin{aligned}
  u_t + g(u)_x = 0,& &&(x,t)\in \R^-\times(0,T),\\
  u(x,0) = u_0(x),& &&x\in \R^-
\end{aligned}
\label{conservation law on R-}
\end{gather}
on $\R^-$ with the flux $g$ being strictly monotone and consider entropy solutions in the following sense.
\begin{definition}[Entropy solution on $\R^-$]\label{def: entropy solution on R-}
  We say $u\in\mathcal{C}([0,T];\mathrm{L}^1(\R^-))\cap\\\mathrm{L}^\infty((0,T)\times\R^-)$ is an entropy solution of Equation~\eqref{conservation law on R-} if for all $c\in\R$,
  \begin{multline*}
    \int_0^T \int_{\R^-} (|u-c|\varphi_t + |g(u)-g(c)|\varphi_x) \diff x\diff t -\int_{\R^-}|u(x,T)-c|\varphi(x,T)\diff x \\
    +\int_{\R^-}|u_0(x)-c|\varphi(x,0))\diff x - \int_0^T |g(u(0-,t))-g(c)|\varphi(0,t)\diff t \geq 0
  \end{multline*}
  for all nonnegative $\varphi\in\mathcal{C}^\infty((-\infty,0]\times[0,T])$.
\end{definition}
Note that here $u(0-,t)$ denotes the limit of $u(x,t)$ as $x\to0$ from the left.



As before, we will write $\cell_j = (x_\jmhf,x_\jphf)$, $j\in\Z$, where now $x_{-\hf}:=0$.
%
Our numerical scheme then reads
\begin{gather}
\begin{aligned}
  u_j^{n+1} = u_j^n - \lambda \left( g(u_j^n) - g(u_{j-1}^n) \right),& &&j < 0,~n\geq 0,\\
  u_j^0 = \frac{1}{\Dx} \int_{\cell_j} u_0(x)\diff x,& &&j< 0
\end{aligned}
\label{numerical scheme on R-}
\end{gather}
where $\lambda=\Dt/\Dx$ satisfies the CFL condition~\eqref{CFL condition}.







We note that the numerical scheme satisfies a discrete entropy inequality away from the spatial boundary
\begin{equation}
  \Dplust \eta_j^n + \Dminus q_j^n \leq 0,\qquad n\geq 1,~j< 0
  \label{discrete entropy inequality on R-}
\end{equation}
which can be seen by adopting the classical Crandall--Majda arguments in~\cite[Prop. 4.1]{crandall1980monotone} for $j< 0$.
Here, $\eta_j^n=\eta(u_j^n,c)=|u_j^n-c|$, $q_j^n = q(u_j^n,c)=\sign(u_j^n -c)(g(u_j^n)-g(c))=|g(u_j^n)-g(c)|$ and
\begin{equation*}
  \Dplust a^n = \frac{a^{n+1}-a^n}{\Dt}\qquad\text{and}\qquad \Dminus a_j = \frac{a_j-a_{j-1}}{\Dx}
\end{equation*}
denote standard difference operators.

In order to derive convergence rates we will develop a Kuznetsov-type lemma in the following.
For any function $u\in\mathcal{C}([0,T];\rm{L}^1(\R^-))$ we define
\begin{align*}
  L(u,c,\varphi) =& \int_0^T\int_{\R^-} \left( |u-c| \varphi_t +  q(u,c)\varphi_x \right)\diff x\diff t -\int_{\R^-} |u(x,T)-c|\varphi(x,T)\diff x\\
  &+ \int_{\R^-} |u_0(x)-c|\varphi(x,0)\diff x - \int_0^T q(u(0-,t),c)\varphi(0,t)\diff t
\end{align*}
where $q(u,c)=|g(u)-g(c)|$ is the Kru\v{z}kov entropy flux. Note that if $u$ is an entropy solution of~\eqref{conservation law on R-} then $L(u,c,\varphi)\geq 0$ for all $c\in\R$ and test functions $\varphi\geq 0$. We now take $c=v(y,s)$ and the test function
\begin{equation*}
  \varphi(x,t,y,s) = \omega_{\varepsilon}(x-y)\omega_{\varepsilon_0}(t-s)
\end{equation*}
where $\omega_{\varepsilon},\omega_{\varepsilon_0}$ are standard symmetric mollifiers for $\varepsilon,\varepsilon_0>0$. Note that $\varphi_t = -\varphi_s$, $\varphi_x=-\varphi_y$ and
\begin{equation}
  \varphi(x,t,y,s) = \varphi(y,t,x,s) = \varphi(y,s,x,t) = \varphi(x,s,y,t)
  \label{properties of varphi}
\end{equation}
as well as
\begin{align}
\begin{aligned}
  \int_{\R}\omega_{\varepsilon}(x-y)\diff y &\leq 1,\\
  \int_0^T \omega_{\varepsilon_0}(t-s)\diff s  &\leq 1,
\end{aligned}
&&
\begin{aligned}
  \int_{\R} |\omega_{\varepsilon}'(x-y)|\diff y &\leq \frac{C}{\varepsilon},\\
  \int_0^T |\omega_{\varepsilon_0}'(t-s)|\diff s &\leq \frac{C}{\varepsilon_0}
\end{aligned}
\label{properties of the mollifier}
\end{align}
for all $x\in\R$, $t\in [0,T]$.
Let now
\begin{equation}
  \Lambda_{\varepsilon,\varepsilon_0}(u,v) = \int_0^T\int_{\R^-} L(u,v(y,s),\varphi(\cdot,\cdot,y,s))\diff y\diff s.
  \label{Definition of Lambda}
\end{equation}
For functions $w\in\mathcal{C}([0,T];\mathrm{L}^1(\R^-))$, we further define the moduli of continuity
\begin{align*}
  \nu_t(w,\varepsilon_0) &= \sup_{|\sigma|\leq \varepsilon_0} \|w(\cdot,t+\sigma)-w(\cdot,t)\|_{\mathrm{L}^1(\R^-)},\\
  \mu(w(\cdot,t),\varepsilon) &= \sup_{|z|\leq \varepsilon} \|w(\cdot+z,t)-w(\cdot,t)\|_{\mathrm{L}^1(\R^-)}.
\end{align*}

\begin{lemma}[Kuznetsov-type lemma]\label{Lemma: Kuznetsov}
  Let $u$ be the entropy solution of~\eqref{conservation law on R-}. Then, for any function $v:[0,T]\to(\Lone\cap\BV)(\R^-)$ such that the one-sided limits $v(t\pm)$ exist in $\Lone$, we have
  \begin{multline*}
    \|u(\cdot,T) - v(\cdot,T)\|_{\mathrm{L}^1(\R^-)} \\+ \int_0^T\int_{\R^-}\int_0^T\left( q(u(0-,t),v(y,s)) + q(v(0-,t),u(y,s)) \right)\varphi(0,t,y,s)\diff t\diff y\diff s\\
    \shoveleft{\leq \|u_0 - v(\cdot,0)\|_{\mathrm{L}^1(\R^-)} - \Lambda_{\varepsilon,\varepsilon_0}(v,u)}\\
    + C\bigg( \varepsilon +\varepsilon_0 + \nu_T(v,\varepsilon_0) + \nu_0(v,\varepsilon_0) + \mu(v(\cdot,T),\varepsilon)) + \mu(v(\cdot,0),\varepsilon)) \bigg)
  \end{multline*}
  for some constant $C$ independent of $\varepsilon$ and $\varepsilon_0$.
\end{lemma}

\begin{proof}
  Using that $\varphi_t = -\varphi_s$, $\varphi_x=-\varphi_y$ and the symmetry relations~\eqref{properties of varphi} we get
  \begin{align*}
    &\Lambda_{\varepsilon,\varepsilon_0}(u,v)\\
    &= -\Lambda_{\varepsilon,\varepsilon_0}(v,u)\\
     &\mathrel{\phantom{=}}- \underbrace{\int_0^T\int_{\R^-}\int_{\R^-}\left( |u(x,T)-v(y,s)|+|v(x,T)-u(y,s)| \right)\varphi(x,T,y,s)\diff x\diff y\diff s}_{=:\mathbf{A}}\\
    &\mathrel{\phantom{=}}+ \underbrace{\int_0^T\int_{\R^-}\int_{\R^-}\left( |u_0(x,t)-v(y,s)|+|v(x,0)-u(y,s)| \right)\varphi(x,0,y,s)\diff x\diff y\diff s}_{=:\mathbf{B}}\\
    &\mathrel{\phantom{=}}- \underbrace{\int_0^T\int_{\R^-}\int_0^T\left( q(u(0-,t),v(y,s)) + q(v(0-,t),u(y,s)) \right)\varphi(0,t,y,s)\diff t\diff y\diff s}_{=:\mathbf{C}}.
  \end{align*}
  Since $u$ is an entropy solution we find
  \begin{equation*}
    0\leq \Lambda_{\varepsilon,\varepsilon_0}(u,v) = -\Lambda_{\varepsilon,\varepsilon_0}(v,u) -\mathbf{A}+\mathbf{B}-\mathbf{C}
  \end{equation*}
  and thus
  \begin{equation*}
    \mathbf{A} +\mathbf{C} \leq -\Lambda_{\varepsilon,\varepsilon_0}(v,u) +\mathbf{B}.
  \end{equation*}
  The terms $\mathbf{A}$ and $\mathbf{B}$ also appear in the case of an unbounded spatial domain and can be estimated by
  \begin{align*}
    \mathbf{A} \geq& \|u(\cdot,T) - v(\cdot,T)\|_{\mathrm{L}^1(\R^-)}\\
    &- \frac{1}{2}\left( \nu_T(u,\varepsilon_0) + \mu(u(\cdot,T),\varepsilon) + \nu_T(v,\varepsilon_0) + \mu(v(\cdot,T),\varepsilon) \right)
    \intertext{and}
    \mathbf{B} \leq& \|u_0 - v(\cdot,0)\|_{\mathrm{L}^1(\R^-)} + \frac{1}{2}\left( \nu_0(u,\varepsilon_0) + \mu(u_0,\varepsilon) + \nu_0(v,\varepsilon_0) + \mu(v(\cdot,0),\varepsilon) \right),
  \end{align*}
  see \cite{coclite2017convergent} or \cite{holden2015front} for details.
  Lastly, due to the Lipschitz continuity in time and the TVD property (see \cite[Thm. 2.15]{holden2015front} and \cite[Lem. A.1]{holden2015front}) the entropy solution of~\eqref{conservation law on R-} satisfies
  \begin{align*}
    \nu_0(u,\varepsilon_0),~\nu_T(u,\varepsilon_0) &\leq C \TV(u_0) \varepsilon_0\\
    \text{and}\qquad \mu(u_0,\varepsilon),~\mu(u(\cdot,T),\varepsilon) &\leq \TV(u_0)\varepsilon
  \end{align*}
  which completes the proof.
\end{proof}

In order to derive a convergence rate the next step is to estimate the term $\Lambda_{\varepsilon,\varepsilon_0}(u_\Dt,u)$.

\begin{lemma}\label{Lemma: Estimate on Lambda}
  The estimate
  \begin{equation*}
    -\Lambda_{\varepsilon,\varepsilon_0}(u_\Dt,u) \leq C\left( \Dx + \frac{\Dx}{\varepsilon} + \frac{\Dt}{\varepsilon_0} \right)
  \end{equation*}
  holds for some constant $C$ independent of $\Dx,\Dt,\varepsilon$, and $\varepsilon_0$.
\end{lemma}
\begin{proof}
	The proof of Lemma~\ref{Lemma: Estimate on Lambda} for conservation laws on the real line can be found e.g. in \cite{holden2015front}. Here, we only need to replace any sum of the form $\sum_{j=-\infty}^\infty$ by $\sum_{-\infty}^{-1}$ and note that the boundary term in space cancels after integration by parts. Confer also the proof of Lemma~\ref{Lemma: Estimate on Lambda 2} in Section~\ref{subsec: convergence rates on R+} for details.
\end{proof}

\begin{theorem}[Convergence rate on $\R^-$]\label{thm: convergence rate on R-}
  Let $u$ be the entropy solution of the initial-boundary value problem~\eqref{conservation law on R-} and $u_\Dt$ the numerical approximation given by~\eqref{numerical scheme on R-} where we take $\lambda$ constant. Then we have the following convergence rate estimate:
  \begin{equation*}
    \|u(\cdot,T) - u_\Dt(\cdot,T)\|_{\mathrm{L}^1(\R^-)} \leq C \sqrt{\Dx}
  \end{equation*}
  for some constant $C$ independent of $\Dx$.
\end{theorem}
\begin{proof}
  The numerical solution $u_\Dt$ is Lipschitz continuous in time and TVD, and therefore satisfies
  \begin{align*}
    \nu_0(u_\Dt,\varepsilon_0),\,\nu_T(u_\Dt,\varepsilon_0) &\leq C\TV(u_0)(\varepsilon_0 +\Dt)\\
    \text{and}\qquad \mu(u_\Dt(\cdot,0),\varepsilon),~\mu(u_\Dt(\cdot,T),\varepsilon) &\leq \TV(u_0) \varepsilon.
  \end{align*}
  Thus, taking into consideration Lemmas~\ref{Lemma: Kuznetsov} and~\ref{Lemma: Estimate on Lambda}, we have
  \begin{multline*}
    \|u(\cdot,T) - u_\Dt(\cdot,T)\|_{\mathrm{L}^1(\R^-)}\\
    + \int_0^T\int_{\R^-}\int_0^T\left( q(u(0-,t),u_\Dt(y,s)) + q(u_\Dt(0-,t),u(y,s)) \right)\varphi(0,t,y,s)\diff t\diff y\diff s\\
    \leq \|u_0 - u_\Dt(\cdot,0)\|_{\mathrm{L}^1(\R^-)} + C\left(\Dx + \Dt +\varepsilon +\varepsilon_0 + \frac{\Dx}{\varepsilon} + \frac{\Dx}{\varepsilon_0} + \frac{\Dt}{\varepsilon_0} \right).
  \end{multline*}
  Because of our choice of discretizing the initial datum as $u_j^0 = \frac{1}{\Dx}\int_{\cell_j}u_0(x)\diff x$ we have $\|u_\Dt(\cdot,0)-u_0\|_{\mathrm{L}^1(\R^-)} \leq C\TV(u_0)\Dx$.
  Now, in order to get a convergence rate, we take $\lambda = \frac{\Dt}{\Dx}$ constant and minimize the right-hand side of the above estimate for $\varepsilon$ and $\varepsilon_0$. This yields $\varepsilon=\varepsilon_0=\sqrt{\Dx}$ and hence
  \begin{multline}
    \|u(\cdot,T) - u_\Dt(\cdot,T)\|_{\mathrm{L}^1(\R^-)} \\
    + \int_0^T\int_{\R^-}\int_0^T\left( q(u(0-,t),u_\Dt(y,s)) + q(u_\Dt(0-,t),u(y,s)) \right)\varphi(0,t,y,s)\diff t\diff y\diff s\\
    \leq C \sqrt{\Dx}.
    \label{Almost convergence rate on R-}
  \end{multline}
  Using the monotonicity of $g$ we find
  \begin{equation*}
    q(u,v) = |g(u)-g(v)|\geq 0
  \end{equation*}
  and thus the integral term in~\eqref{Almost convergence rate on R-} is nonnegative which concludes the proof.
\end{proof}

\subsection{Convergence rate estimates on $\R^+$}\label{subsec: convergence rates on R+}
As a second step we now consider the initial-boundary value problem
\begin{gather}
\begin{aligned}
  u_t + f(u)_x = 0,& &&(x,t)\in \R^+\times(0,T),\\
  u(x,0) = u_0(x),& &&x\in \R^+,\\
  u(0,t) = f^{-1}\left(g(u(0-,t)\right),& &&t\in (0,T)
\end{aligned}
\label{conservation law on R+}
\end{gather}
and the numerical scheme
\begin{gather}
\begin{aligned}
  u_j^{n+1} = u_j^n - \lambda \left( f(u_j^n) - f(u_{j-1}^n) \right),& &&j\geq 1,~n\geq 0\\
  u_j^0 = \frac{1}{\Dx} \int_{\cell_j} u_0(x)\diff x,& &&j\geq 0,\\
  u_0^n = f^{-1}\left(g(u_{-1}^n)\right),& &&n\geq 1
\end{aligned}
\label{numerical scheme on R+}
\end{gather}
where the boundary data is given in terms of $u(0,-,t)$ and $u_{-1}^n$ respectively and those are known from the previous section. Note that again we have a discrete entropy inequality of the form
\begin{equation}
  \Dplust \eta_j^n + \Dminus q_j^n \leq 0,\qquad n\geq 1,~j\geq 1.
  \label{discrete entropy inequality on R+}
\end{equation}

\begin{definition}[Entropy solution on $\R^+$]\label{def: entropy solution on R+}
  We say $u\in\mathcal{C}([0,T];\mathrm{L}^1(\R^+))\cap\\\mathrm{L}^\infty(\R^+\times (0,T))$ is an entropy solution of Equation~\eqref{conservation law on R+} if for all $c\in\R$,
  \begin{multline*}
    \int_0^T \int_{\R^+} (|u-c|\varphi_t + |f(u)-f(c)|\varphi_x) \diff x\diff t -\int_{\R^+}|u(x,T)-c|\varphi(x,T)\diff x\\
    +\int_{\R^+}|u_0(x)-c|\varphi(x,0))\diff x + \int_0^T |f(u(0+,t))-f(c)|\varphi(0,t)\diff t \geq 0
  \end{multline*}
  for all nonnegative $\varphi\in\mathcal{C}^\infty([0,\infty)\times[0,T])$ and
  \begin{equation*}
    f(u(0+,t)) = g(u(0-,t))
  \end{equation*}
  holds for almost every $t\in(0,T)$.
\end{definition}

Before we calculate convergence rates on $\R^+$ we need two auxiliary lemmas that are consequences of the monotonicity of the flux.

\begin{lemma}[Bound on the temporal total variation]\label{Lemma: Bound on temporal variation}
  If the numerical scheme \eqref{numerical scheme on R+} satisfies the CFL condition~\eqref{CFL condition} the temporal variation of the numerical solution is bounded, specifically, for every $j\in\Z$ we have
  \begin{equation*}
    \sum_{n=0}^M |u_j^{n+1} - u_j^n| \leq C \TV(u_0)
  \end{equation*}
  where $\TV(u_0)$ refers to the total variation of $u_0$ on the whole real line.
\end{lemma}

\begin{proof}
  Let first $j\geq 1$. Using the CFL condition~\eqref{CFL condition} and the monotonicity of the flux, i.e., $f'>0$, we find that
  \begin{align*}
    |u_j^n - u_{j-1}^n - \lambda (f(u_j^n) - f(u_{j-1}^n))| &= |u_j^n - u_{j-1}^n - \lambda  f'(u^*)(u_j^n - u_{j-1}^n)|\\
    &=(1-\lambda  f'(u^*))|u_j^n - u_{j-1}^n|\\
    &= |u_j^n - u_{j-1}^n| - \lambda  f'(u^*)|u_j^n - u_{j-1}^n|\\
    &= |u_j^n - u_{j-1}^n| - \lambda  |f(u_j^n) - f(u_{j-1}^n)|
  \end{align*}
  and hence
  \begin{align*}
    |u_j^{n+1} - u_{j-1}^{n+1}| &= |u_j^n - u_{j-1}^n - \lambda (f(u_j^n) - f(u_{j-1}^n)) + \lambda  (f(u_{j-1}^n) - f(u_{j-2}^n))|\\
    &\leq |u_j^n - u_{j-1}^n - \lambda (f(u_j^n) - f(u_{j-1}^n))| + \lambda  |f(u_{j-1}^n) - f(u_{j-2}^n)|\\
    &=|u_j^n - u_{j-1}^n| - \lambda |f(u_j^n) - f(u_{j-1}^n)| + \lambda  |f(u_{j-1}^n) - f(u_{j-2}^n)|\\
    &=|u_j^n - u_{j-1}^n| - |u_j^{n+1} - u_j^n| + |u_{j-1}^{n+1} - u_{j-1}^n|
  \end{align*}
  where we have used the definition of the numerical scheme~\eqref{numerical scheme on R+} in the last step.
  Taking the sum over $n=0,\ldots,M-1$ yields
  \begin{equation*}
    \sum_{n=0}^{M-1} |u_j^{n+1}-u_{j-1}^{n+1}| \leq \sum_{n=0}^{M-1} |u_j^n - u_{j-1}^n| - \sum_{n=0}^{M-1} |u_j^{n+1} - u_j^n| +  \sum_{n=0}^{M-1} |u_{j-1}^{n+1} - u_{j-1}^n|
  \end{equation*}
  where we can cancel equal terms to get
  \begin{equation}
    |u_j^M-u_{j-1}^M| \leq  |u_j^0 - u_{j-1}^0| - \sum_{n=0}^{M-1} |u_j^{n+1} - u_j^n| +  \sum_{n=0}^{M-1} |u_{j-1}^{n+1} - u_{j-1}^n|.
    \label{bound on u^M}
  \end{equation}
  Because of the CFL condition~\eqref{CFL condition} we have
  \begin{equation*}
  	|u_j^{M+1}-u_j^M| = \lambda  |f(u_j^M)-f(u_{j-1}^M)| = \lambda f'(u^*)  |u_j^M - u_{j-1}^M| \leq |u_j^M-u_{j-1}^M|
  \end{equation*}
  which together with~\eqref{bound on u^M} yields
  \begin{equation*}
  	|u_j^{M+1} - u_j^M| \leq |u_j^0 - u_{j-1}^0| - \sum_{n=0}^{M-1} |u_j^{n+1} - u_j^n| +  \sum_{n=0}^{M-1} |u_{j-1}^{n+1} - u_{j-1}^n|
  \end{equation*}
  and thus
  \begin{equation}
  	\sum_{n=0}^M |u_j^{n+1} - u_j^n| \leq |u_j^0-u_{j-1}^0| + \sum_{n=0}^{M-1}|u_{j-1}^{n+1} - u_{j-1}^n|.
  	\label{estimate for the temporal variation}
  \end{equation}
  By substituting $f$ with $g$ in the above calculations, the estimate \eqref{estimate for the temporal variation} also holds for $j<0$.
  The estimate \eqref{estimate for the temporal variation} now allows us to bound the temporal variation of the numerical scheme by the total variation of the initial datum in the following way. If $j>M$ or $j<0$ we can use the estimate~\eqref{estimate for the temporal variation} iteratively to get
  \begin{equation*}
  	\sum_{n=0}^M |u_j^{n+1} - u_j^n| \leq \sum_{i=j-M+1}^j |u_{i}^0 - u_{i-1}^0| + |u_{j-M}^1 - u_{j-M}^0|.
  \end{equation*}
  Using the definition of the scheme~\eqref{numerical scheme on R+}, we get
  \begin{equation*}
  	|u_{j-M}^1 - u_{j-M}^0| = \lambda |f(u_{j-M}^0) - f(u_{j-M-1}^0)| \leq C \lambda |u_{j-M}^0 - u_{j-M-1}^0|
  \end{equation*}
  such that we have
  \begin{equation*}
  	\sum_{n=0}^M |u_j^{n+1} - u_j^n| \leq C \sum_{i=j-M}^j |u_{i}^0 - u_{i-1}^0| \leq C \TV(u_0).
  \end{equation*}
  If on the other hand $0\leq j\leq M$ we get
  \begin{equation*}
  	\sum_{n=0}^M |u_j^{n+1} - u_j^n| \leq \sum_{i=1}^j |u_{i}^0 - u_{i-1}^0| + \sum_{n=0}^{M-j}|u_0^{n+1} - u_0^n|.
  \end{equation*}
  Using the definition of $u_0^n$ in~\eqref{numerical scheme on R+} and applying~\eqref{estimate for the temporal variation} iteratively again, we get
  \begin{align*}
  	\sum_{n=0}^{M-j} |u_0^{n+1}-u_0^n| &= \sum_{n=0}^{M-j} \left| f^{-1}\left(g(u_{-1}^{n+1})\right)- f^{-1}\left(g(u_{-1}^n)\right) \right|\\
  	&\leq \frac{C}{\alpha} \sum_{n=0}^{M-j} |u_{-1}^{n+1} - u_{-1}^n|\\
  	&\leq C \mkern-10mu\sum_{i=-1-(M-j)}^{-1} \mkern-10mu|u_i^0 - u_{i-1}^0|
  \end{align*}
  such that we have
\begin{equation*}
	\sum_{n=0}^M |u_j^{n+1} - u_j^n| \leq C \mkern-10mu\sum_{i=-1-(M-j)}^j \mkern-10mu|u_{i}^0 - u_{i-1}^0| \leq C \TV(u_0)
\end{equation*}
which concludes the proof.
\end{proof}

\begin{lemma}\label{Lemma: Lipschitz continuity in space of f(u)}
  Let $u$ be the entropy solution of~\eqref{conservation law on R+} and assume $f'>0$. Then $f(u)$ is Lipschitz continuous in space, in the sense that
  \begin{equation*}
    \int_0^T |f(u(x,t)) - f(u(y,t))|\diff t \leq C|x-y|\qquad \text{for all }x,y\in\R^+.
  \end{equation*}
\end{lemma}

\begin{proof}
  Since $u$ is bounded, we can assume that $f'\geq \alpha>0$. Thus the flux is invertible with Lipschitz continuous inverse. By setting $w=f(u)$ and $h=f^{-1}$ we find that $w$ satisfies
  \begin{equation*}
    w_x + h(w)_t =0,\qquad (t,x)\in(0,T)\times\R^+.
  \end{equation*}
  By the standard theory for conservation laws (with the roles of $x$ and $t$ reversed) adapted to the bounded domain $[0,T]$ we see that $w$ is Lipschitz continuous in $x$ with values in $\Lone(0,T)$, i.e.,
  \begin{equation*}
    \int_0^T |f(u(x,t)) - f(u(y,t))|\diff t = \int_0^T |w(x,t) - w(y,t)|\diff t \leq C|x-y|,
  \end{equation*}
  cf. \cite[Thm. 2.15]{holden2015front} or \cite[Lem. 4]{Ridder2019}. Note that an application of \cite[Lem. 4]{Ridder2019} requires, in particular, a temporal total variation bound of $u(0+,t)$ which follows from Lemma~\ref{Lemma: Bound on temporal variation} on a discrete level and caries over in the limit.
\end{proof}

We will now describe how to modify the steps in Section~\ref{subsec: Convergence rates on R-} in order to get a convergence rate on $\R^+$. We start by defining
\begin{align*}
  L(u,c,\varphi) =& \int_0^T\int_{\R^+} \left( |u-c| \varphi_t + q(u,c)\varphi_x \right)\diff x\diff t -\int_{\R^+} |u(x,T)-c|\varphi(x,T)\diff x\\
  &+ \int_{\R^+} |u_0(x)-c|\varphi(x,0)\diff x + \int_0^T q(u(0+,t),c)\varphi(0,t)\diff t
\end{align*}
and
\begin{equation*}
  \Lambda_{\varepsilon,\varepsilon_0}(u,v) = \int_0^T\int_{\R^+} L(u,v(y,s),\varphi(\cdot,\cdot,y,s))\diff y\diff s
\end{equation*}
where again $\varphi=\omega_{\varepsilon}(x-y)\omega_{\varepsilon_0}(t-s)$.
\begin{lemma}[Kuznetsov-type lemma]\label{Lemma: Kuznetsov 2}
  Let $u$ be the entropy solution of~\eqref{conservation law on R+}. Then, for any function $v:[0,T]\to (\Lone\cap\BV)(\R^+)$ such that the one-sided limits $v(t\pm)$ exist in $\Lone$, we have
  \begin{multline*}
    \|u(\cdot,T) - v(\cdot,T)\|_{\mathrm{L}^1(\R^+)} \leq \|u_0 - v(\cdot,0)\|_{\mathrm{L}^1(\R^+)} - \Lambda_{\varepsilon,\varepsilon_0}(v,u)\\
    + C\bigg( \varepsilon +\varepsilon_0 + \nu_T(v,\varepsilon_0) + \nu_0(v,\varepsilon_0) + \mu(v(\cdot,T),\varepsilon)) + \mu(v(\cdot,0),\varepsilon)) \bigg) \\
    + \int_0^T\int_{\R^+}\int_0^T\left( q(u(0+,t),v(y,s)) + q(v(0+,t),u(y,s)) \right)\varphi(0,t,y,s)\diff t\diff y\diff s
  \end{multline*}
  for some constant $C$ independent of $\varepsilon$ and $\varepsilon_0$.
\end{lemma}
Note that this time the term involving $q$ is on the right hand side of the inequality.
\begin{proof}
	The proof follows the same steps, \textit{mutatis mutandis}, as the proof of the Kuznetsov-type lemma~\ref{Lemma: Kuznetsov} on $\R^-$.
\end{proof}

\begin{lemma}\label{Lemma: Estimate on Lambda 2}
  The estimate
  \begin{equation*}
    -\Lambda_{\varepsilon,\varepsilon_0}(u_\Dt,u) \leq C\left( \Dx + \frac{\Dx}{\varepsilon} + \frac{\Dx}{\varepsilon_0} + \frac{\Dt}{\varepsilon_0} \right)
  \end{equation*}
  holds for some constant $C$ independent of $\Dx,\Dt,\varepsilon$, and $\varepsilon_0$.
\end{lemma}

Note that the right-hand side of the inequality contains the term $\frac{\Dx}{\varepsilon_0}$ which was not present in Lemma~\ref{Lemma: Estimate on Lambda}, but will not change the overall convergence rate.
\begin{proof}
Using summation by parts and the discrete entropy inequality~\eqref{discrete entropy inequality on R+}, $\Dplust \eta_j^n + \Dminus q_j^n \leq 0$ for $j\geq 1$, we find
\begin{align*}
    &-L(u_{\Delta t},u,\varphi)\\
    &= - \sum_{n=0}^M\sum_{j=0}^\infty \left(\eta_j^n \iint_{\cell_j^n}\varphi_t\diff x\diff t +  q_j^n \iint_{\cell_j^n}\varphi_x \diff x\diff t \right) \\
    &\mathrel{\phantom{=}} - \sum_{j=0}^\infty \eta_j^0\int_{\cell_j}\varphi^0\diff x + \sum_{j=0}^\infty\eta_j^{M+1}\int_{\cell_j}\varphi^{M+1}\diff x - \sum_{n=0}^M q_0^n\int_{\cell^n}\varphi_{-\hf}\diff t \\
    &= \sum_{n=0}^M \sum_{j=0}^\infty \left(\Dplust\eta_j^n\iint_{\cell_j^n} \varphi^{n+1} \diff x\diff t+  \Dminus q_{j+1}^n\iint_{\cell_j^n} \varphi_{j+\frac{1}{2}} \diff x\diff t \right)\\
    &\leq \sum_{n=0}^M \Bigg(\Dplust\eta_0^n \iint_{\cell_0^n} \varphi^{n+1}\diff x\diff t + \sum_{j=0}^\infty  \Dminus q_{j+1}^n \iint_{\cell_j^n}\varphi_\jphf\diff x\diff t \\
    &\mathrel{\phantom{=}}\phantom{\sum_{n=0}^N\Bigg(}- \sum_{j=1}^\infty  \Dminus q_j^n \iint_{\cell_j^n}\varphi^{n+1}\diff x\diff t \Bigg)\\
    &= \sum_{n=0}^M \underbrace{\Dplust\eta_0^n \iint_{\cell_0^n} \varphi^{n+1}\diff x\diff t}_{=:\mathbf{D^n}} + \sum_{n=0}^M\sum_{j=1}^\infty \underbrace{\Dminus q_j^n \iint_{\cell_j^n}(\varphi_\jmhf - \varphi^{n+1})\diff x\diff t}_{=:\mathbf{E_j^n}}
  \end{align*}
where we have used the notation $\varphi^n = \varphi(x,t^n,y,s)$ and $\varphi_\jphf = \varphi(x_\jphf,t,y,s)$.
%
Concerning the term $\mathbf{D^n}$, using summation by parts again, we find
\begin{multline*}
  \sum_{n=0}^M \int_0^T \int_{\R^+} \mathbf{D^n} \diff y\diff s\\
  = \int_0^T\int_{\R^+} \left( \eta_0^{M+1}\int_{\cell_0}\varphi^{M+1}\diff x - \eta_0^0\int_{\cell_0}\varphi^0 \diff x - \sum_{n=0}^M \eta_0^n \iint_{\cell_0^n}\Dplust \varphi^n\diff x\diff t \right)\diff y\diff s.
\end{multline*}
Here, using the boundedness of $\eta'$ and the properties of the mollifiers~\eqref{properties of the mollifier}, the boundary terms can be estimated as follows:
\begin{equation*}
  \int_0^T \int_{\R^+} \underbrace{\eta_0^{M+1}}_{\leq C\|u_0\|_\infty} \int_{\cell_0}\varphi^{M+1}\diff x\diff y\diff s \leq C \Dx
\end{equation*}
and similarly
\begin{equation*}
  \int_0^T \int_{\R^+} \eta_0^0 \int_{\cell_0}\varphi^0\diff x\diff y\diff s \leq C \Dx.
\end{equation*}
For the remaining term, we can proceed in the following way
\begin{align*}
  \sum_{n=0}^M &\int_0^T\int_{\R^+} \underbrace{\eta_0^n}_{\leq C \|u_0\|_\infty} \iint_{\cell_0^n} \Dplust \varphi^n \diff x\diff t\diff y\diff s\\
  &\leq C  \sum_{n=0}^M\int_0^T\int_{\R^+} \frac{1}{\Dt}\iint_{\cell_0^n}\int_{\cell^n}|\omega_{\varepsilon_0}'(\tau-s)|\diff \tau \omega_{\varepsilon}(x-y)\diff x\diff t\diff y\diff s\\
  &\leq C\sum_{n=0}^M \frac{\Dx\Dt^2}{\Dt\varepsilon_0}\\
  &\leq C\frac{\Dx}{\varepsilon_0}.
\end{align*}
%
We split the term involving $\mathbf{E_j^n}$ as follows:
\begin{align*}
  \sum_{n=0}^M \sum_{j=1}^\infty \mathbf{E_j^n} \leq& \sum_{n=0}^M \sum_{j=1}^\infty \underbrace{|\Dminus q_j^n| \iint_{\cell_j^n}\int_{x_\jphf}^x |\varphi_x(z,t)|\diff z\diff x\diff t}_{=:\mathbf{F_j^n}}\\
  &+ \sum_{n=0}^M \sum_{j=1}^\infty \underbrace{|\Dminus q_j^n| \iint_{\cell_j^n}\int_t^{t^{n+1}} |\varphi_t(x,\tau)|\diff \tau\diff x\diff t}_{=:\mathbf{G_j^n}}.
\end{align*}
For the first term, using the properties of the mollifiers~\eqref{properties of the mollifier} and the Lipschitz continuity of $f$, we find
\begin{align*}
  \int_0^T\int_{\R^+} &\sum_{n=0}^M\sum_{j=1}^\infty \mathbf{F_j^n} \diff y\diff s\\
  &= \sum_{n=0}^M\sum_{j=1}^\infty |\Dminus q_j^n| \int_0^T\int_{\R^+}\iint_{\cell_j^n}\int_{x_\jphf}^x |\omega_{\varepsilon}'(z-y)|\diff z \omega_{\varepsilon_0}(t-s)\diff x\diff t\diff y\diff s\\
  &\leq \sum_{n=0}^M \sum_{j=1}^\infty \frac{C}{\Dx}|u_j^n-u_{j-1}^n| \frac{\Dx^2\Dt}{\varepsilon}\\
  &\leq C \TV(u_0) \frac{\Dx}{\varepsilon}
\end{align*}
and similarly
\begin{equation*}
  \int_0^T\int_{\R^+} \sum_{n=0}^M\sum_{j=1}^\infty \mathbf{G_j^n} \diff y\diff s \leq C \TV(u_0) \frac{\Dt}{\varepsilon_0}.
\end{equation*}
Thus, we have
\begin{equation*}
  -\Lambda_{\varepsilon,\varepsilon_0}(u_\Dt,u) \leq C \left( \Dx + \frac{\Dx}{\varepsilon} + \frac{\Dx}{\varepsilon_0} + \frac{\Dt}{\varepsilon_0} \right)
\end{equation*}
which concludes the proof.
\end{proof}

\begin{theorem}[Convergence rate on $\R^+$]\label{thm: convergence rate on R+}
  Let $u$ be the entropy solution of the initial-boundary value problem~\eqref{conservation law on R+} and $u_\Dt$ the numerical approximation given by~\eqref{numerical scheme on R+}. Then we have the following convergence rate estimate:
  \begin{equation*}
    \|u(\cdot,T) - u_\Dt(\cdot,T)\|_{\mathrm{L}^1(\R^+)} \leq C \sqrt{\Dx}
  \end{equation*}
  for some constant $C$ independent of $\Dx$.
\end{theorem}
\begin{proof}
  The numerical solution $u_\Dt$ is Lipschitz continuous in time and TVD (for the TVD property of conservation laws on bounded domains see \cite[Lem. 2]{Ridder2019}), and therefore satisfies
  \begin{align*}
    \nu_0(u_\Dt,\varepsilon_0),\,\nu_T(u_\Dt,\varepsilon_0) &\leq C\TV(u_0)(\varepsilon_0 +\Dt)\\
    \text{and}\qquad \mu(u_\Dt(\cdot,0),\varepsilon),~\mu(u_\Dt(\cdot,T),\varepsilon) &\leq \TV(u_0) \varepsilon.
  \end{align*}
  Thus, taking into consideration Lemmas~\ref{Lemma: Kuznetsov} and~\ref{Lemma: Estimate on Lambda}, we have
  \begin{align*}
    \|u(\cdot,T) - u_\Dt(\cdot,T)\|_{\mathrm{L}^1(\R^+)} &\leq \|u_0 - u_\Dt(\cdot,0)\|_{\mathrm{L}^1(\R^+)}\\
    &\mathrel{\phantom{\leq}}+ C\left(\Dx + \Dt +\varepsilon +\varepsilon_0 + \frac{\Dx}{\varepsilon} + \frac{\Dx}{\varepsilon_0} + \frac{\Dt}{\varepsilon_0} \right) + \mathbf{C}
  \end{align*}
  where
  \begin{equation*}
  	\mathbf{C} = \int_0^T\int_{\R^+}\int_0^T \bigg( q(u(0+,t),u_\Dt(y,s)) + q(u_\Dt(0+,t),u(y,s)) \bigg)\varphi(0,t,y,s)\diff t\diff y\diff s.
  \end{equation*}
  Because of our choice of discretizing the initial datum as $u_j^0 = \frac{1}{\Dx}\int_{\cell_j}u_0(x)\diff x$ we have $\|u_\Dt(\cdot,0)-u_0\|_{\mathrm{L}^1(\R^+)} \leq C\TV(u_0)\Dx$ and thus it remains to estimate the term
  \begin{multline*}
    \mathbf{C} = \int_0^T\int_{\R^+}\int_0^T \bigg( \underbrace{ |f(u(0+,t)) -f(u_\Dt(y,s))|}_{=:\mathbf{H}}\\
    + \underbrace{|f(u_\Dt(0+,t)) - f(u(y,s))|}_{=:\mathbf{J}} \bigg)\varphi(0,t,y,s)\diff t\diff y\diff s.
  \end{multline*}
  Here, we split
  \begin{equation*}
    \mathbf{H} \leq \underbrace{|f(u(0+,t)) - f(u_\Dt(0+,s))|}_{\mathbf{H_1}} + \underbrace{|f(u_\Dt(0+,s)) - f(u_\Dt(y,s))|}_{\mathbf{H_2}}.
  \end{equation*}
 Using the Rankine--Hugoniot condition, the term $\mathbf{H_1}$ can be estimated as follows:
\begin{align*}
	\mathbf{H_1} &= |f(u(0+,t)) - f(u_\Dt(0+,s))|\\
	&= |g(u(0-,t)) - g(u_\Dt(0-,s))|\\
	&\leq  |g(u(0-,t)) - g(u_\Dt(y,s))| + |g(u_\Dt(y,s)) -g(u_\Dt(0-,s))|.
\end{align*}
Because $\mathbf{H_1}$ does not depend on $y$ we can use the symmetry of $\varphi$ with respect to $y$ and the estimate~\eqref{Almost convergence rate on R-} to get
  \begin{align*}
  	&\int_0^T\int_{\R^+}\int_0^T \mathbf{H_1} \varphi(0,t,y,s)\diff t\diff y\diff s\\
  	&= \int_0^T\int_{\R^-}\int_0^T \mathbf{H_1} \varphi(0,t,y,s)\diff t\diff y\diff s\\
  	&\leq \int_0^T\int_{\R^-}\int_0^T |g(u(0-,t))-g(u_\Dt(y,s))| \varphi(0,t,y,s)\diff t\diff y\diff s\\
  	&\mathrel{\phantom{=}} + \int_0^T\int_{\R^-}\int_0^T |g(u_\Dt(y,s))-g(u_\Dt(0-,s))| \varphi(0,t,y,s)\diff t\diff y\diff s\\
  	&\leq C\sqrt{\Dx} +  \int_0^T\int_{\R^-}\int_0^T |g(u_\Dt(y,s))-g(u_\Dt(0-,s))| \varphi(0,t,y,s)\diff t\diff y\diff s.
  \end{align*}
  %
  Using the identity
  \begin{equation*}
    |g(u_{i+1}^n) - g(u_i^n)| = \frac{1}{\lambda}|u_{i+1}^{n+1}-u_{i+1}^n|
  \end{equation*}
  and setting $N=\lceil\frac{\varepsilon}{\Dx}\rceil$, we can employ Lemma~\ref{Lemma: Bound on temporal variation} to estimate the integral term in the foregoing estimate as follows:
  \begin{align*}
  	& \int_0^T\int_{\R^-}\int_0^T |g(u_\Dt(y,s))-g(u_\Dt(0-,s))| \varphi(0,t,y,s)\diff t\diff y\diff s\\
  	&= \sum_{n=0}^M\sum_{j=-N}^{-1} |g(u_j^n) - g(u_{-1}^n)| \int_0^T\iint_{\cell_j^n}\omega_{\varepsilon_0}(t-s)\omega_{\varepsilon}(y)\diff t\diff y\diff s\\
  	&\leq C\frac{\Dt\Dx}{\varepsilon} \sum_{n=0}^M \sum_{j=-N}^{-1} \sum_{i=j}^{-2} \underbrace{|g(u_{i+1}^n)-g(u_i^n)|}_{= \frac{1}{\lambda}|u_{i+1}^{n+1}-u_{i+1}^n|}\\
  	&\leq C \frac{\Dt\Dx}{\varepsilon} \sum_{j=-N}^{-1} \sum_{i=j}^{-2} \sum_{n=0}^M |u_{i+1}^{n+1}-u_{i+1}^n|\\
  	&\leq C \frac{\Dt\Dx}{\varepsilon} \sum_{j=-N}^{-1} (-j)\\
  	&\leq C\frac{\Dt\Dx}{\varepsilon} \frac{N(N+1)}{2}\\
  	&\leq C \Dt\left(\frac{\varepsilon}{\Dx}+1\right)\\
  	&\leq C (\varepsilon + \Dt).
  \end{align*}
The term involving $\mathbf{H_2}$ can be estimated analogously.
Then it remains to treat the integral involving $\mathbf{J}$. We split $\mathbf{J}$ as follows
\begin{equation*}
	\mathbf{J} \leq \underbrace{|f(u_\Dt(0+,t)) - f(u(0+,s))|}_{\mathbf{J_1}} + \underbrace{|f(u(0+,s)) - f(u(y,s))|}_{\mathbf{J_2}}
\end{equation*}
and note that the $\mathbf{J_1}$ is the same as $\mathbf{H_1}$.
Lastly, with the help of Lemma~\ref{Lemma: Lipschitz continuity in space of f(u)} we find
  \begin{align*}
    \int_0^T\int_{\R^+}\int_0^T & \mathbf{J_2} \varphi(0,t,y,s)\diff t\diff y\diff s\\
    &= \int_0^T\int_{\R^+}\int_0^T |f(u(0+,s)) - f(u(y,s))| \varphi(0,t,y,s)\diff t\diff y\diff s\\
    &\leq \frac{1}{\varepsilon} \int_0^\varepsilon \underbrace{\int_0^T |f(u(0+,s)) - f(u(y,s))| \diff s}_{\leq C|y|}\diff y\\
    &\leq \frac{C}{\varepsilon}\int_0^{\varepsilon}|y|\diff y\\
    &\leq C\varepsilon.
  \end{align*}
%
%
  Finally, we have
  \begin{equation*}
    \|u(\cdot,T) - u_\Dt(\cdot,T)\|_{\mathrm{L}^1(\R^+)} \leq C\left(\Dx + \Dt +\varepsilon +\varepsilon_0 + \frac{\Dx}{\varepsilon} + \frac{\Dx}{\varepsilon_0} + \frac{\Dt}{\varepsilon_0} \right).
  \end{equation*}
  In order to get a convergence rate, again we take $\lambda = \frac{\Dt}{\Dx}$ constant and minimize the right-hand side of the above estimate for $\varepsilon$ and $\varepsilon_0$. This yields $\varepsilon=\varepsilon_0=\sqrt{\Dx}$ which concludes the proof.
\end{proof}

\subsection{Convergence rate estimates on $(0,L)$}

By restricting the solution $u$ and the numerical approximation $u_\Dt$ to a bounded interval $(0,L)$ Theorem~\ref{thm: convergence rate on R+} and the estimate~\eqref{Almost convergence rate on R-} yield a convergence rate on $(0,L)$. Note that this is only possible since $f$ is strictly monotone.
\begin{corollary}[Convergence rate on $(0,L)$]\label{cor: convergence rate on (0,L)}
  Let $u$ be the entropy solution of the initial-boundary value problem~\eqref{conservation law on R+} on the bounded interval $[0,L]$ and $u_\Dt$ the numerical approximation given by~\eqref{numerical scheme on R+}. Then we have the following convergence rate estimate:
  \begin{equation*}
    \|u(\cdot,T) - u_\Dt(\cdot,T)\|_{\mathrm{L}^1(0,L)} \leq C \sqrt{\Dx}
  \end{equation*}
  for some constant $C$ independent of $\Dx$.
\end{corollary}
\begin{proof}
  Without repeating all calculations of Sections~\ref{subsec: Convergence rates on R-} and~\ref{subsec: convergence rates on R+} we will highlight the adjustments to the respective proofs that need to be done. If we consider solutions on $(0,L)$ instead of $\R^+$ the definition of $\Lambda_{\varepsilon,\varepsilon_0}(u,v)$ in~\ref{Definition of Lambda} needs to be adjusted so that $\Lambda_{\varepsilon,\varepsilon_0}(u,v)$ contains the term
  \begin{equation*}
    -\int_0^T\int_0^L\int_0^T q(u(L-,t),v(y,s))\varphi(L,t,y,s) \diff t\diff y\diff s
  \end{equation*}
  and all instances of $\R^+$ need to be changed to $(0,L)$. 
  Following the proofs of Theorems~\ref{thm: convergence rate on R-} and~\ref{thm: convergence rate on R+} in the same way finally yields
  \begin{multline}
    \|u(\cdot,T) - u_\Dt(\cdot,T)\|_{\mathrm{L}^1(0,L)} \\
    + \int_0^T\int_0^L\int_0^T\left( q(u(L,t),u_\Dt(y,s)) + q(u_\Dt(L,t),u(y,s)) \right)\varphi(L,t,y,s)\diff t\diff y\diff s\\
    \leq C \sqrt{\Dx}.
    \label{Almost convergence rate on (0,L)}
  \end{multline}
  Using the monotonicity of $f$ we find
  \begin{equation*}
    q(u,v) = |f(u)-f(v)|\geq 0
  \end{equation*}
  and thus the integral term in~\eqref{Almost convergence rate on (0,L)} is nonnegative which concludes the proof.
\end{proof}

\section{Statement and proof of the main theorem}\label{sec: statement and proof of the main theorem}
Our main result now reads as follows:
\begin{theorem}[Convergence rate for conservation laws with discontinuous flux]\label{thm: convergence rate for conservation laws with discontinuous flux}
Let $u$ be the entropy solution of Equation~\eqref{conservation law} and $u_\Dt$ the numerical solution given by~\eqref{numerical scheme on R}. Then we have the following convergence rate:
\begin{equation*}
  \|u(\cdot,T) - u_\Dt(\cdot,T)\|_{\mathrm{L}^1(\R)} \leq C \sqrt{\Dx}
\end{equation*}
for some constant $C$ independent of $\Dx$.
\end{theorem}
\begin{proof}
As before, we decompose the entropy solution $u$ as $u=\sum_{i=0}^N u^{(i)}$ where $u^{(i)}$, $i=0,\ldots,N$, are the respective entropy solutions on $D_i$, i.e., solutions of~\eqref{conservation law on D0} and~\eqref{conservation law on Di} respectively. Further, we decompose the numerical solution $u_\Dt$ as $\sum_{i=0}^N u^{(i)}_\Dt$ where
\begin{equation*}
  u^{(i)}_\Dt(x,t)=\begin{cases}
    u_j^n &\text{if } (x,t)\in\cell_j^n\subset D_i\times\cell^n,\\
    0 &\text{otherwise}
  \end{cases}
\end{equation*}
and $u_j^n$ is given by~\eqref{numerical scheme on R}.
Then we have
\begin{equation*}
 \|u(T) - u_{\Dt}(T)\|_{\mathrm{L}^1(\R)} = \sum_{i=0}^N \|u^{(i)}(T) - u^{(i)}_\Dt(T)\|_{\mathrm{L}^1(D_i)}.
\end{equation*}
Using Theorem~\ref{thm: convergence rate on R-} for $D_0$, Theorem~\ref{thm: convergence rate on R+} for $D_N$, and Corollary~\ref{cor: convergence rate on (0,L)} for each $D_i$, $i=1,\ldots,N-1$, shows that
\begin{equation*}
	\|u^{(i)}(T) - u^{(i)}_\Dt(T)\|_{\mathrm{L}^1(D_i)} \leq C\sqrt{\Dx}
\end{equation*}
for $i=0,\ldots,N$ which concludes the proof.
\end{proof}

\begin{remark}
  Note that the rate of Theorem~\ref{thm: convergence rate for conservation laws with discontinuous flux} is optimal in the sense that it can not be improved without further assumptions on the initial datum. This can easily be shown in the same way as in the absence of spatial dependency since the specific initial datum $u_0$ constructed by \c{S}abac in \cite{sabac} can be chosen in a way such that $u_0$ is supported away from the last discontinuity.
\end{remark}

\section{Numerical experiments}\label{sec: Numerical experiments}
To illustrate our results we now present two numerical experiments. We consider the `two flux' case
\begin{gather*}
\begin{aligned}
  u_t + (H(x)f(u) + (1-H(x))g(u))_x = 0,& &&(x,t)\in \R\times(0,T),\\
  u(x,0) = u_0(x),& &&x\in \R
\end{aligned}
\end{gather*}
where $H$ is the Heaviside function. This corresponds to switching from one $u$-dependent flux, $g$, to another, $f$.
\paragraph{Experiment 1}
In our first numerical experiment we choose $g(u)=u$ and $f(u)=\unitfrac{u^2}{2}$ such that we switch from the transport equation to the Burgers equation across $x=0$. The initial datum we consider for Experiment~$1$ is
\begin{equation*}
	u_0(x)= \begin{cases}
		0.5 &\text{if }x<-0.5,\\
		2 &\text{if }x>-0.5
	\end{cases}
\end{equation*}
which is chosen such that the Rankine--Hugoniot condition at $x=0$ gives $u(0-,t)=u(0+,t)$ before the jump at $x=-0.5$ interacts with the interface. Figure~\ref{fig: Numerical experiments 1} shows the numerical solution calculated with the scheme~\eqref{numerical scheme on R} with open boundaries in blue and the initial datum in gray (dashed line) at various times (before, during, and after interaction with the interface). We used $\Dx=\unitfrac{2}{n}$ with $n=64$, end time $T=0.9$, and $\lambda=0.5$. We clearly recognize the characteristic features of the transport equation and the Burgers equation here as the upward jump in the initial datum is transported to the right as a shock until it crosses the interface at $x=0$ where the shock, as it enters the Burgers regime, subsequently becomes a rarefaction wave.
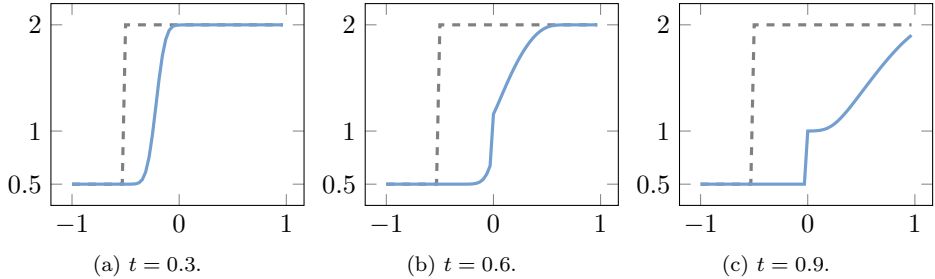
\begin{figure}
\centering
\subfloat[$t=0.3$.]{
\begin{tikzpicture}
	\begin{axis}[ymin=0.3,ymax=2.2,xtick={0,32,64},xticklabels={$-1$,$0$,$1$},ytick={0.5,1,2},yticklabels={$0.5$,$1$,$2$}]
   		\addplot[gray, dashed,very thick,mark=none] table {ic.txt};
		\addplot[skyblue1,very thick,mark=none] table {AdvectionToBurgersThird.txt};
	\end{axis}
\end{tikzpicture}
}
\subfloat[$t=0.6$.]{
\begin{tikzpicture}
	\begin{axis}[ymin=0.3,ymax=2.2,xtick={0,32,64},xticklabels={$-1$,$0$,$1$},ytick={0.5,1,2},yticklabels={$0.5$,$1$,$2$}]
   		\addplot[gray, dashed,very thick,mark=none] table {ic.txt};
		\addplot[skyblue1,very thick,mark=none] table {AdvectionToBurgersTwoThird.txt};
	\end{axis}
\end{tikzpicture}
}
\subfloat[$t=0.9$.]{
\begin{tikzpicture}
	\begin{axis}[ymin=0.3,ymax=2.2,xtick={0,32,64},xticklabels={$-1$,$0$,$1$},ytick={0.5,1,2},yticklabels={$0.5$,$1$,$2$}]
   		\addplot[gray, dashed,very thick,mark=none] table {ic.txt};
		\addplot[skyblue1,very thick,mark=none] table {AdvectionToBurgersEnd.txt};
	\end{axis}
\end{tikzpicture}
}
\caption{Numerical solution of Experiment $1$ with $\Dx=\unitfrac{2}{64}$ at various times.}
\label{fig: Numerical experiments 1}
\end{figure}
\paragraph{Experiment 2}
In our second numerical experiment we choose $g(u)=\unitfrac{u^2}{2}$ and $f(u)=u$ such that we switch from the Burgers equation to the transport equation across $x=0$. The initial datum we consider is
\begin{equation*}
	u_0(x) = 2+\operatorname{exp}(-100(x+0.75)^2).
\end{equation*}
Again, the offset of the initial datum is chosen in a way such that the Rankine--Hugoniot condition at $x=0$ gives $u(0-,t)=u(0+,t)$ before the non-constant part of $u_0$ interacts with the interface. Figure~\ref{fig: Numerical experiments 2} shows the numerical solution calculated with the scheme~\eqref{numerical scheme on R} with open boundaries in blue and the initial datum in gray (dashed line) at various times (immediately before, during, and after interaction with the interface). We used $\Dx=\unitfrac{2}{n}$ with $n=128$, end time $T=0.5$, and $\lambda=0.2$. We clearly recognize the shock formation due to the Burgers regime to the left of the interface (see Figure~\ref{fig: Numerical experiments 2} (a)).
Note that -- although difficult to see in Figure~\ref{fig: Numerical experiments 2} (c) because of numerical diffusion -- the shock is preserved over the interface (only with a different profile).
\begin{figure}
\centering
\subfloat[$t= 0.2$.]{
\begin{tikzpicture}
	\begin{axis}[ymin=1.8,ymax=3.5,xtick={0,64,128},xticklabels={$-1$,$0$,$1$},ytick={2,3},yticklabels={$\phantom{0.}2$,$3$}]
   		\addplot[gray, dashed,very thick,mark=none] table {ic2.txt};
		\addplot[skyblue1,very thick,mark=none] table {BurgersToAdvectionHalf.txt};
	\end{axis}
\end{tikzpicture}
}
\subfloat[$t=0.3$.]{
\begin{tikzpicture}
	\begin{axis}[ymin=1.8,ymax=3.5,xtick={0,64,128},xticklabels={$-1$,$0$,$1$},ytick={2,3},yticklabels={$\phantom{0.}2$,$3$}]
   		\addplot[gray, dashed,very thick,mark=none] table {ic2.txt};
		\addplot[skyblue1,very thick,mark=none] table {BurgersToAdvectionTwoThird.txt};
	\end{axis}
\end{tikzpicture}
}
\subfloat[$t=0.5$.]{
\begin{tikzpicture}
	\begin{axis}[ymin=1.8,ymax=3.5,xtick={0,64,128},xticklabels={$-1$,$0$,$1$},ytick={2,3},yticklabels={$\phantom{0.}2$,$3$}]
   		\addplot[gray, dashed,very thick,mark=none] table {ic2.txt};
		\addplot[skyblue1,very thick,mark=none] table {BurgersToAdvectionEnd.txt};
	\end{axis}
\end{tikzpicture}
}
\caption{Numerical solution of Experiment $2$ with $\Dx=\unitfrac{2}{128}$ at various times.}
\label{fig: Numerical experiments 2}
\end{figure}
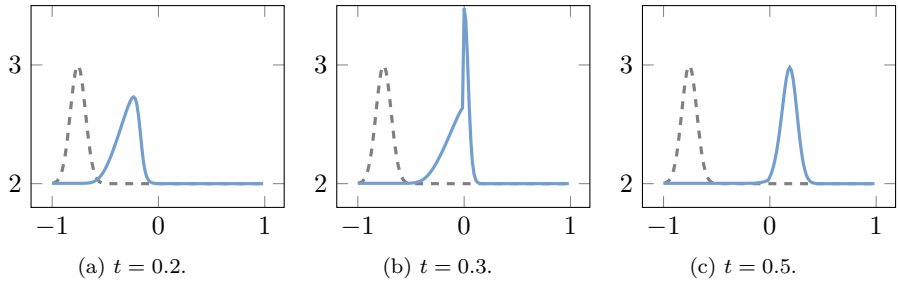

Table~\ref{tab: Convergence rates} shows the observed convergence rates of the solution at time $T=0.9$  for Experiment $1$ and at time $T=0.5$ for Experiment $2$ for various values of $\Dx$. As a reference solution, we used a numerical solution on a very fine grid ($n=2048$) in both cases. As expected from experience in the case of spatially independent flux we observe convergence rates strictly between $\unitfrac{1}{2}$ and $1$ (cf. e.g.~\cite[clawpack software]{leveque2002finite}).

\begin{table}
  \centering
  \subfloat[Experiment $1$.]{
  \begin{tabular}{rcccc}
  \toprule
  \multicolumn{1}{c}{$n$} & $\Lone$ error & $\Lone$ OOC\\
  \midrule
  $ 16$ &  $\num{1.751e-01}$ & -- \\
  $ 32$ &  $\num{1.256e-01}$ & $0.48$ \\
  $ 64$ &  $\num{8.865e-02}$ & $0.50$ \\
  $128$ &  $\num{5.918e-02}$ & $0.58$ \\
  $256$ &  $\num{3.637e-02}$ & $0.70$ \\
  $512$ &  $\num{1.978e-02}$ & $0.88$ \\
  $1024$&  $\num{8.145e-03}$ & $1.28$ \\
  \bottomrule
\end{tabular}
  }
  \hspace{2em}
  \subfloat[Experiment $2$.]{
  \begin{tabular}{rcccc}
  \toprule
  \multicolumn{1}{c}{$n$} & $\Lone$ error & $\Lone$ OOC\\
  \midrule
  $ 16$ &  $\num{2.771e-01}$ & -- \\
  $ 32$ &  $\num{1.823e-01}$ & $0.60$ \\
  $ 64$ &  $\num{1.261e-01}$ & $0.53$ \\
  $128$ &  $\num{8.390e-02}$ & $0.59$ \\
  $256$ &  $\num{5.125e-02}$ & $0.71$ \\
  $512$ &  $\num{2.780e-02}$ & $0.88$ \\
  $1024$&  $\num{1.132e-02}$ & $1.30$ \\
  \bottomrule
\end{tabular}
  }
  \caption{Convergence rates of Experiment $1$ and $2$.}
  \label{tab: Convergence rates}
\end{table}

\section{Conclusion}\label{sec: Conclusion}
Scalar conservation laws with discontinuous flux frequently occur in physical applications and several numerical schemes have been considered in the literature. In contrast to the case where the nonlinear flux does not have a spatial dependency, however, convergence rate results for monotone finite volume schemes have not been available until now.

In this paper, we have established a convergence rate for upwind-type finite volume methods for the case where $f$ is strictly monotone in $u$ and the spatial dependency $k$ is piecewise constant with finitely many discontinuities. The central idea of this paper is to split the problem into finitely many conservation laws between two neighboring discontinuities of $k$ and thus get a convergence rate as a consequence of convergence rates on bounded domains.
Here, the novel feature of this paper is the strong bound on the temporal total variation of the finite volume approximation which allows us to estimate the boundary terms in space at the discontinuities of $k$ that appear when applying the classical Kuznetsov theory to problem~\eqref{conservation law}.

As an outlook we name four possible directions of future research. A first direction would be to extend the convergence rate result of this paper to the cases where $k$ is not piecewise constant and $f$ is not monotone. Second, it might be interesting to investigate convergence rates of monotone schemes in the Wasserstein distance. In the case of spatially independent fluxes, convergence rates in the Wasserstein distance are well-understood due to Nessyahu, Tadmor and Tassa \cite{nessyconv92,nessyconv}. 
A third direction of future research might be to see whether the results of this paper can be extended to monotone schemes in conservation form, i.e., where the definition of $u_{P_i}^{n+1}$ in \eqref{numerical scheme on R} is replaced by $u_{P_i}^{n+1} = u_{P_i}^n -\lambda(f^{(i)}(u_{P_i}^n) - f^{(i-1)}(u_{P_i -1}^n))$.
Lastly, convergence rates of the front tracking method for conservation laws with discontinuous flux are highly desirable as well. In the case of spatially independent fluxes, convergence rates of the front tracking method are known in $\mathrm{L}^1$ due to Lucier \cite{lucier1986moving} and in the Wasserstein distances due to Solem \cite{solem2018convergence}.



\section*{Acknowledgments}
We thank Nils Henrik Risebro for several useful discussions and Ulrik Skre Fjordholm for his careful reading of the manuscript. We also like to thank the referees for their constructive and insightful comments.

\appendix
\section{Convergence rate estimates for general initial-boundary value problems}\label{app: Convergence rates for general IBVPs}
With the techniques developed in this paper, we can also derive a convergence rate for the initial-boundary value problem
\begin{gather}
\begin{aligned}
  u_t + f(u)_x = 0,& &&(x,t)\in (0,L)\times(0,T),\\
  u(x,0) = u_0(x),& &&x\in (0,L),\\
  u(0,t) = a(t),& &&t\in (0,T)
\end{aligned}
\label{conservation law in appendix}
\end{gather}
and the numerical scheme
\begin{gather*}
\begin{aligned}
  u_j^{n+1} = u_j^n - \lambda \left( f(u_j^n) - f(u_{j-1}^n) \right),& &&j\geq 1,~n\geq 0\\
  u_j^0 = \frac{1}{\Dx} \int_{\cell_j} u_0(x)\diff x,& &&j\geq 0,\\
  u_0^n = \frac{1}{\Dt} \int_{\cell^n} a(s)\diff s,& &&n\geq 1.
\end{aligned}
\end{gather*}
Here we need to assume that $a\in(\mathrm{L}^1\cap\BV)(0,T)$ which allows us to use the total variation of $a$ directly instead of crossing the discontinuity in Lemma~\ref{Lemma: Bound on temporal variation}. The assertion of Lemma~\ref{Lemma: Bound on temporal variation} should then read
\begin{equation*}
  \sum_{n=0}^M |u_j^{n+1} - u_j^n| \leq C(\TV(u_0) + \TV(a))
\end{equation*}
which can be used at the same place Lemma~\ref{Lemma: Bound on temporal variation} is used in Theorem~\ref{thm: convergence rate on R+}. Hence, Corollary~\ref{cor: convergence rate on (0,L)} gives the convergence rate $\mathcal{O}(\sqrt{\Dx})$ for the general initial-boundary value problem~\eqref{conservation law in appendix}.
Note that this is a higher rate than the $\mathcal{O}(\Dx^{\unitfrac{1}{3}})$ rate mentioned in \cite{ohlberger2006error}.

\bibliographystyle{siamplain}

\end{document}